\documentclass[reqno,11pt,a4paper]{amsart}

\usepackage{amsfonts,amsmath,amssymb,amsthm,color}

\usepackage{amssymb}
\usepackage{amsfonts}
\usepackage{amsmath}
\usepackage{amsthm}
\usepackage{nccmath}
 \usepackage{mathrsfs}
\usepackage[USenglish]{babel}
\usepackage{esint}
\usepackage{graphicx}
\usepackage[colorlinks=true]{hyperref}
\usepackage[latin1]{inputenc}
\usepackage[hmargin=3cm, vmargin=2.9cm]{geometry}
\hypersetup{
    colorlinks,
    citecolor=red,
    linkcolor=blue,
    urlcolor=black
}

\def\R {\mathbb{R}}

\def\N {\mathbb{N}}

\def\d{\delta}
\def\x{\xi}

\newcommand{\pa}{\partial}
\newcommand{\bs}[1]{\boldsymbol{#1}}
\newcommand{\mf}[1]{\mathbf{#1}}

\newtheorem{proposition}{Proposition}[section]
\newtheorem{theorem}[proposition]{Theorem}
\newtheorem*{theorem*}{Theorem}

\newtheorem{lemma}[proposition]{Lemma}

\theoremstyle{definition}
\newtheorem{remark}[proposition]{Remark}
\numberwithin{equation}{section}

\begin{document}
\title[A cubic Schr\"odinger system with mixed attractive and repulsive forces]{Solutions to a cubic Schr\"odinger system with mixed attractive and repulsive forces in a critical regime}

\author{Simone Dovetta}
\address[S. Dovetta]{Dipartimento di Scienze di Base e Applicate per l'Ingegneria, Sapienza Universit\`a di Roma, Via Scarpa 16, 00161 Roma, Italy}
\email{simone.dovetta@uniroma1.it}
 
\author{Angela Pistoia}
\address[A. Pistoia]{Dipartimento di Scienze di Base e Applicate per l'Ingegneria, Sapienza Universit\`a di Roma, Via Scarpa 16, 00161 Roma, Italy}
\email{angela.pistoia@uniroma1.it}

\begin{abstract}
We study the existence of solutions to the cubic Schr\"odinger system
$$
-\Delta u_i =  \sum_{j =1}^m \beta_{ij} u_j^2u_i  + \lambda_i u_i\ \hbox{in}\ \Omega,\ u_i=0\ \hbox{on}\ \partial\Omega,\ i =1,\dots,m,
$$
when $\Omega$ is a bounded domain in $\mathbb R^4, $ $\lambda_i$ are positive small numbers, $\beta_{ij}$ are real numbers so that $\beta_{ii}>0$ and $\beta_{ij}=\beta_{ji}$, $i\neq j$. We assemble the components $u_i$ in groups so that all the interaction forces $\beta_{ij}$ among components of the same group are attractive, i.e. $\beta_{ij}>0$, while forces among components of different groups are repulsive or weakly attractive, i.e. $\beta_{ij}<\overline\beta$ for some $\overline\beta$ small.
We   find solutions  such that
  each component within a given group  blows-up around the same point  and the different groups blow-up around different points, as all the parameters $\lambda_i$'s approach zero.

\end{abstract}

\date\today
 \thanks{  }

\maketitle

\section{Introduction}
The study of solitary waves $\Phi_i=\mathtt{exp}(\iota \omega_i t)u_i$ of the nonlinear Schr\"odinger system
$$-\iota\partial_t\Phi_i=\Delta \Phi_i+\Phi_i\sum\limits_{j=1}^m\beta_{ij}|\Phi_j|^2,\ \Phi_i: \Omega\to\mathbb C,\ i=1,\dots,m,$$
where $\Omega$ is a   smooth domain in $\mathbb R^N$ naturally leads to study 
the elliptic system
\begin{equation}\label{soli}-\Delta u_i +\omega_i u_i=  \sum_{j =1}^m \beta_{ij} u_j^2u_i ,\ u_i: \Omega\to\mathbb R,\ i=1,\dots,m.\end{equation}
Here $\omega_i$ and $\beta_{ij}=\beta_{ji}$ are real numbers and   $\beta_{ii}>0.$
This type of systems arises in many physical models such as incoherent wave packets in Kerr medium in
nonlinear optics (see \cite{AkAn}) and in Bose--Einstein condensates for multi--species condensates (see \cite{Timm}).
The coefficient  $\beta_{ij}$  represents the interaction force between components $u_i$ and $u_j$. The sign of $\beta_{ij}$ determines
whether the interactions between components are {\em repulsive} (or {\em competitive}), i.e. $\beta_{ij}<0$, or {\em attractive} (or {\em cooperative}), i.e. $\beta_{ij}>0$. 
In particular, one usually assumes $\beta_{ii}>0.$ We observe that system \eqref{soli} has always the trivial solution, namely when all the components  vanish. If one or more components are identically zero, then system \eqref{soli} reduces to a system with a smaller number of components. Therefore, we are interested in finding solutions whose all components are not trivial. These are called 
{\em fully nontrivial} solutions.
\\

In low dimensions  $1\le N\le 4$, problem \eqref{soli} has a variational structure: solutions to \eqref{soli} are critical points of the energy $J:H\to\mathbb R$
defined by
$$J(u):=\frac12\sum\limits_{i=1}^m\int\limits_\Omega\left(|\nabla u_i|^2+\omega_iu_i^2\right)-\frac14\sum\limits_{i,j=1}^m\beta_{ij}\int\limits_\Omega u_i^2u_j^2,$$
where the space $ H$ is either $H^1(\Omega)$ or $H^1_0(\Omega),$ depending on the boundary conditions associated to $u_i$ in \eqref{soli} in the case of not empty $\partial\Omega$.
Therefore, the existence and multiplicity of solutions can be obtained using classical methods in critical point theory. 
However, there is an important difference between the dimensions $1\le N\le 3$ and the dimension $N=4$. Actually, in dimension $N=4$ the nonlinear part of $J$ has a {\em critical} growth and the lack of compactness of the Sobolev embedding $H^1(\Omega)\hookrightarrow L^4(\Omega)$ makes  difficult the search for critical points.
On the other hand, in dimensions $1\le N\le 3$ the problem has a {\em subcritical} regime and the variational tools can be successfully applied to get a wide number of results. We refer to the introduction of the most recent paper \cite{bks} for an
overview on the topic and for a complete list of references.
Up to our knowledge, the higher dimensional case $N\ge5$ is completely open, because the problem does not have a  variational structure and new ideas are needed.\\

In this paper, we will focus on problem \eqref{soli} when $\Omega$ is a smooth bounded domain in $\R^4$ with Dirichlet boundary condition. We shall rewrite \eqref{soli} in the form
\begin{equation}\label{system}
-\Delta u_i =  \sum_{j =1}^m \beta_{ij} u_j^2u_i  + \lambda_i u_i\ \hbox{in}\ \Omega,\ u_i=0\ \hbox{on}\ \partial\Omega,\ i =1,\dots,m,
\end{equation}
where $\lambda_i$ are real numbers,
as this way it can be seen as a generalization   of the celebrated {\em Brezis--Nirenberg} problem \cite{BrezisNirenberg}
\begin{equation}\label{bn}
-\Delta u = u^3  + \lambda u \ \hbox{in}\ \Omega,\
u=0\ \hbox{on}\ \partial\Omega.
\end{equation}
It is worthwhile to remind that the existence of solutions to \eqref{bn} strongly depends on the geometry of $\Omega$.
In particular, if $\Omega$ is a starshaped domain, then Pohozaev's identity  ensures that \eqref{bn} has no solution when $\lambda\le0.$ 
On the other hand, Brezis and Nirenberg \cite{BrezisNirenberg} proved that \eqref{bn} has a positive solution if and only if $\lambda\in(0,\Lambda_1(\Omega))$ where $\Lambda_1$ is the first eigenvalue of $-\Delta$ with homogeneous Dirichlet condition on $\partial\Omega.$ These solutions are often referred to as {\em least energy} solutions, as they can be obtained also by minimizing the functional 
\[
\frac{1}{2}\int_\Omega\left(|\nabla u|^2-\lambda|u|^2\right)-\frac{1}{4}\int_\Omega|u|^4
\]
restricted to the associated Nehari manifold. Later,  Han \cite{HanAIHP} and 
Rey   \cite{Rey} studied the asymptotic behaviour of this solution as $\lambda\to0$ and proved that it  {\em blows--up} at a point $\xi_0 \in \Omega$ which is a critical point of the  Robin's function, whereas far away from $\xi_0$ his shape resembles the {\em bubble} 
\begin{equation}\label{def bubble}
U_{\d,\x}(x):= \alpha  \frac{\delta}{\delta^2 + |x-\xi|^2},\ \alpha=2\sqrt2.
\end{equation}
Recall that it is well known (see \cite{Aubin,Talenti}) that the family $\{U_{\d,\x}: \ \d >0,\ \x \in \R^4\}$ contains all the positive solutions to the critical problem
\begin{equation}\label{critical single}
-\Delta U = U^3  \ \hbox{in}\ \mathbb R^4 .
\end{equation}
Let us also remind that    the Robin's function is defined by $\mathtt r(x):=H(x,x)$, $x\in \Omega,$ where  $H(x,y)$ is the regular part of 
 the Green function of $-\Delta$ in $\Omega$ with Dirichlet boundary condition.

Successively, relying on the profile of the bubble as a first order approximation, the Ljapunov--Schmidt procedure has been fruitfully used to build both positive and sign--changing solutions to \eqref{bn} blowing--up at different points in $\Omega$ as the parameter $\lambda$ approaches zero
(see for example Rey \cite{Rey} and Musso and Pistoia \cite{MussoPistoiaIndiana2002}).
\\

As far as we know, few results are available about existence and multiplicity of solutions to the critical system \eqref{system}. The first result is due to 
Chen and Zou \cite{ChenZou1}, who considered \eqref{system} with 2 components only 
\begin{equation}\label{system2}\left\{\begin{aligned}
&-\Delta u_1 = \mu_1 u_1^3+\beta u_1u_2^2+\lambda_1 u_1\quad \hbox{in}\ \Omega\\
&-\Delta u_2 = \mu_2 u_2^3+\beta u_1^2u_2+\lambda_2 u_2\quad \hbox{in}\ \Omega\\
 & u_1=u_2=0\quad \hbox{on}\ \partial\Omega.\end{aligned}\right.
\end{equation}
When $0<\lambda_1,\lambda_2<\Lambda_1(\Omega)$,
they proved the existence of a least energy positive solution in the competitive  case (i.e. $\beta<0$) and in the cooperative case (i.e. $\beta>0$)  if $\beta\in(0,\underline\beta]\cup[\overline\beta,+\infty)$, for some $\overline\beta\ge\max\{\mu_1,\mu_2\}>\min\{\mu_1,\mu_2\}\ge \underline\beta>0.$ In the cooperative case, when $\lambda_1=\lambda_2$ the least energy solution is  {\em synchronized}, i.e. $(u_1,u_2)=(c_1u,c_2u)$ where $u$ is the least energy positive solution of the equation \eqref{bn} and $(c_1,c_2)$ is a positive solution to the algebraic system
$$\left\{\begin{aligned}
&1 = \mu_1 c_1^2+\beta c_2^2\\
&1 = \mu_2 c_2^2+\beta c_1^2.\end{aligned}\right.
$$
In the competitive case,  the authors studied also the limit profile of the components of the least energy solution and proved that the following alternative occurs: either one of the components vanishes and the other one converges to a least energy positive solution of the equation \eqref{bn}, or both components survive and their limits separate in different regions of the domain $\Omega$, i.e. a {\em phase separation} phenomenon takes place.
In the subcritical regime such a phenomenon has been studied by Noris, Tavares, Terracini and Verzini \cite{NTTV}.

Afterwards, Chen and Lin \cite{ChenLin} studied the asymptotic behavior of the least energy solution of \eqref{system2} in the cooperative case as $\max\{\lambda_1,\lambda_2\}\to0$ and found that both components blow--up at the same critical point of the Robin's function, in the same spirit of  the result by Han and Rey for the single equation \eqref{bn}. 

The existence of blowing--up solutions for system \eqref{system} with an arbitrary number of components has been studied by
Pistoia and Tavares   \cite{PistoiaTavares}. Using a Ljapunov--Schmidt procedure, they built solutions to \eqref{system} whose $m$ components  blow--up at $m$ different non--degenerate critical points of the Robin's function as $\lambda^*:=\max \{\lambda_1,\dots,\lambda_m\}\to0$, provided
the interaction forces are not too large, namely $\beta^*:=\max\limits_{ij}\beta_{ij}\le \overline \beta $ for some $\overline \beta >0.$ 
For example, their result holds in dumbbell shaped domains which are obtained by connecting $m$ mutually disjoint connected domains $D_1,\dots,D_m$ by thin handles.  In this case the Robin's function has $m$ distinct critical points which are non--degenerate for a generic choice of the domain as proved by Micheletti and Pistoia \cite{MiPi}. 
Moreover, if, as $\lambda^*\to0$, we let  $\beta^*:=\max\limits_{i,j}\beta_{ij}$ approach $-\infty$ with a {\em sufficiently low velocity} (depending  on $\lambda^*$), then it is still possible to show that  all the components   blow--up at different points and a   segregation phenomen occurs.\\
To conclude the state of the art, we would like to mention some recent results obtained by exploiting a variational point of view. Guo, Luo and Zou \cite{GLZ}  proved the existence of a least energy solution to \eqref{system}   in the purely cooperative regime (i.e. $\min_{i\not=j}\beta_{ij}\ge0$) when $\lambda_1=\dots=\lambda_m$ and showed that such a solution is synchronized under some additional technical conditions on the coupling coefficients. Tavares and You \cite{TavaresYou} generalized the previous result to a mixed competitive/weakly cooperative regime (i.e.    $\max_{i\not=j}\beta_{ij}$ not too large).   Clapp and Szulkin \cite{ClSz} found a least energy solution  in the purely competitive regime  (i.e.    $\max_{i\not=j}\beta_{ij}<0$), which is not synchronized when  the coupling terms $\beta_{ij}$ diverge to $-\infty.$\\

Now, let us go back to the result obtained by Pistoia and Tavares  \cite{PistoiaTavares} concerning the existence of solutions to \eqref{system} with all the components
blowing--up around different  points in $\Omega$ when all the mixed forces are repulsive or weakly attractive. It is natural to ask what happens for more general mixed  repulsive and attractive forces. 
Our idea is to assemble the components $u_i$ in groups so that all the interaction forces $\beta_{ij}$ among components of the same group are attractive, while forces among components of different groups are repulsive or weakly attractive. In this setting, 
we address the following
question: 
\begin{itemize}
\item[\bf(Q)] {\em is it possible to find solutions  such that
  each component within a given group   concentrates around the same point  and different groups concentrate around different points?}
\end{itemize}
Given  $1 < q < m$, let us introduce a $q-$decomposition of $m$, namely a vector $ (l_0,\dots,l_q) \in \N^{q+1}$ such that
\[
0=l_0<l_1<\dots<l_{q-1}<l_q=m.
\]
Given a $q$--decomposition   of $m$, we set, for $h=1,\dots,q$,
$$
  I_h:= \{i \in  \{1,\dots,m\}:  l_{h-1} < i \le l_h \}.
$$
In this way, we have partitioned the set $\{1,\ldots, m\}$ into $q$ groups $I_1,\ldots, I_q$, and we can consequently split the components of our system into $q$ groups $\{u_i:\ i\in I_h\}$. 
Notice that if $l_{h}-l_{h-1} = 1$, then $I_h$ reduces to the singleton $\{i\}$, for some $i\in\{1,\dots,m\}$. 
 We will assume that   
for every $h=1,\dots,q$ 
\begin{itemize}
\item[(A1)]  the algebraic system 
\begin{equation}\label{sist2intro}
 1= \sum\limits_{j\in I_h}\beta_{ij}  c _j^ 2\,,  \qquad  i \in I_h,
\end{equation}
has a   solution $\mathfrak c_h=(c_i)_{i\in I_h}$ with $c _i>0$ for every $i\in I_h$;
\item[(A2)] the matrix $\left(\beta_{ij}\right)_{i,j\in I_h}$ is invertible and all the entries are positive.
 \end{itemize}
We observe that (A1) is satisfied for instance if for every $i\neq j$ (see  \cite{Bartsch})
$$\beta_{ij}=:\beta>\max\limits_{i\in I_h}\beta_{ii}\ \hbox{for every}\ i\in I_h.$$
From a PDE point of view,   assumption $(A1)$  is equivalent to require that 
the nonlinear PDE system
\begin{equation}\label{sist1intro}
-\Delta W_i=W_i \sum\limits_{j \in I_h } \beta_{ij} W_j^2  \quad  \hbox{in}\ \mathbb R^n,\quad\  i \in I_h,
\end{equation}
 has a {\em synchronized} solution $W_i=c _i U,$ $i\in I_h$, where the positive function
$$U (x):= \alpha  \frac{1}{1 + |x|^2},\ \alpha=2\sqrt2,$$
 solves the critical equation \eqref{critical single}.
 Assumption (A2) ensures that such a synchronized solution of  \eqref{sist1intro} is {\em non--degenerate} (see \cite[Proposition 1.4]{PistoiaSoave}), in the sense that the   linear system (obtained by linearizing system \eqref{sist1intro} around the synchronized solution)
\begin{equation}\label{sist3intro}
-\Delta v_i=U^2\left[\left(3\beta_{ii}c _i^2+ \sum\limits_{j \in I_h \atop j\not=i}\beta_{ij}  c _j^ 2 \right) v_i+ 2\sum\limits_{j \in I_h \atop j\not=i}\beta_{ij}c _i  c _j  v_j\right] \   \text{in }\mathbb R^n, \quad i \in I_h,
\end{equation}
has  a $5$--dimensional set of solutions
\begin{equation}\label{non-de}
\left(v_1,\dots,v_{|I_h|}\right)\in {\mathtt{span}}\left\{\mathfrak e_h\psi^\ell\ |\ \ell=0,1,\dots,4\right\}\subset\left(H^1_0(\Omega)\right)^{|I_h|}
\end{equation}
 where   $\mathfrak e_h\in \mathbb R^{|I_h|}$ is a suitable vector  (see \cite[Lemma 6.1]{PistoiaSoave}) and the functions 
	$$\psi^0(y)={1-|y|^2\over(1+|y|^2)^2}\quad \hbox{and}\quad \psi^\ell(y)={y_\ell\over (1+|y|^2)^2},\ \ell=1,\dots,4,$$
solve the linear equation
$$-\Delta \psi = 3U^2\psi  \quad \hbox{in}\ \mathbb R^4.$$

We are now in position to state our main result.

\begin{theorem}\label{main} Assume (A1) and (A2). Assume furthermore that the Robin's function has $q$ distint non--degenerate critical points $\xi_1^0,\dots,\xi_q^0$. There exist $\overline\beta>0$ and $\lambda_0>0$ such that, if $\beta^*:=\max\limits_{(i,j)\in I_h\times I_k\atop h\not=k}\beta_{ij}<\overline\beta$ then, for every $(\lambda_i)_{i=1}^m$ with $\lambda_i\in(0,\lambda_0)$, $i=1,\dots,m$, there exists a solution $(u_1,\dots,u_m)$ to \eqref{system} such that, for every $h=1,\dots,q$, each group of components $\{u_i\ :\ i\in I_h\}$  blows--up at $\xi_h^0$ as $\lambda^*:=\max\limits_{i=1,\dots,m}\lambda_i\to0.$\\
Moreover, if, as $\lambda^*\to0$,  $\beta^*$ approaches $-\infty$ slowly enough (depending on $\lambda^*$), i.e.
 $|\beta^*|=O\left(e^{d^*\over\lambda^*}\right)\ \hbox{for some $d^*$ sufficiently small}$, then  all the components belonging to different groups  blow--up at different points and segregate, while the components belonging to the same group blow--up at the same point and aggregate.
\end{theorem}

\begin{remark} Theorem \ref{main}   deals with
systems with mixed aggregating and segregating forces (i.e.
some $\beta_{ij}$'s are positive, and some others are negative). This is particularly interesting since
there are few results about  systems with
mixed terms. The subcritical regime has been recently investigated by Byeon, Kwon  and Seok  \cite{bks}, 
	 Byeon, Sato and Wang \cite{BySaWa}, Sato and Wang \cite{SaWa1,SaWa2},
Soave and Tavares \cite{SoaveTavares}, Soave \cite{So} and Wei and Wu \cite{WeiWu}.
As far as we know, there are  only a couple of results concerning the critical regime. The first one has been obtained by  Pistoia and Soave in \cite{PistoiaSoave}, where the authors studied system \eqref{system} when all the $\lambda_i$'s are zero and the domain has some  holes whose size approaches zero. The second one is due to Tavares and You, who in  \cite{TavaresYou} found a least energy solution to system \eqref{system} provided all the parameters $\lambda_i$ are  equal.
\end{remark}

\begin{remark} We strongly believe that the solutions found in Theorem \ref{main} are positive, because
they are constructed as the superposition of positive function and small perturbation term. 
This is true for sure if  the attractive forces $\beta_{ij}$ are small, as proved in \cite{PistoiaTavares}.  
 In the general case, the proof does not work and  some refined  $L^\infty-$estimates of the 
small terms are needed.  We will not afford this issue in the present paper, because the study of the invertibility of the linear operator naturally associated to the problem (see Proposition \ref{linear}) should be performed in spaces equipped with different norms (i.e. $L^\infty$--weighted norms) that may deserve further investigations.
\end{remark}

The proof of Theorem \ref{main}   relies on the well known Ljapunov--Schmidt reduction. The main steps are described in Section \ref{due}, where the details of the proof are omitted whenever it can be obtained, up to minor modifications, by combining the arguments in Pistoia and Tavares \cite{PistoiaTavares} and in Pistoia and Soave \cite{PistoiaSoave}.
Here we limit ourselves to give a detailed proof of the first step of the scheme, as it suggests how to adapt the ideas of \cite{PistoiaTavares,PistoiaSoave} to the present setting. The technical details of this part are developed in the Appendix. Before getting to this, in Section \ref{uno} we recall some well known results that are needed in the following.

\section{Preliminaries}\label{uno}
We denote the standard inner product and norm in $H_0^1(\Omega)$ by 
\[
\langle u, v \rangle_{H_0^1(\Omega)}:= \int_{\Omega} \nabla u \cdot \nabla v, \quad \|u\|_{H_0^1(\Omega)}:= \left( \langle u,u \rangle_{H_0^1(\Omega)}\right)^\frac12,
\]
and the $L^q$-norm ($q \ge 1$) by $|\cdot|_{L^q(\Omega)}$. Whenever the domain of integration $\Omega$ is out of question, we also make use of the shorthand notation $\|u\|$ for $\|u\|_{H_0^1(\Omega)}$ and $|u|_q$ for $|u|_{L^q(\Omega)}$.\\
Let $i:H_0^1(\Omega  ) \to L^{4}(\Omega  )$ be the canonical Sobolev embedding. We consider the adjoint operator $(-\Delta)^{-1}: L^{\frac43}(\Omega  ) \to H_0^1(\Omega  )$ characterized by
\[
(-\Delta)^{-1}(u) = v \quad \iff \quad \begin{cases} -\Delta v= u & \text{in $\Omega$} \\ v \in H_0^1(\Omega) \end{cases}
\]
It is well known that $(-\Delta)^{-1}$ is a continuous operator, and relying on it we can rewrite \eqref{system} as
\begin{equation}\label{pb 1}
u_i = (-\Delta)^{-1}\left(  \sum_{j =1}^m \beta_{ij} u_j^2   u_i+\lambda_iu_i\right),\ \quad\ i=1,\dots,m.
\end{equation}
From now on, we will focus on problem \eqref{pb 1}.\\

 We are going to build a solution $\mf u=(u_1,\dots,u_m)$ to \eqref{pb 1}, whose main term, as the parameters $\lambda_i$ approach zero, is defined in terms of the bubbles $U_{\d,\x}$ given in \eqref{def bubble}. More precisely, let us consider the projection $P  U_{\d,\xi}$ of $U_{\d,\xi}$ into $H_0^1(\Omega )$, i.e. the unique solution to 
$$
-\Delta (P  U_{\d,\xi}) = -\Delta U_{\d,\x} = U_{\d,\x}^3 \ \hbox{in}\ \Omega,\ 
P  U_{\d,\x}  = 0 \ \hbox{on}\ \partial\Omega.$$
We shall use many times the fact that $0 \le P  U_{\d,\x} \le U_{\d,\x}$, which is a simple consequence of the maximum principle. 
Moreover it is well known that
$$PU_{\delta,\xi}(x)=U_{\delta,\xi}(x)-\alpha\delta H(x,\xi)+\mathcal O\left(\delta^{3}\right).$$
Here 
$G(x,y)$ is  the Green function of $-\Delta$ with Dirichlet boundary condition in $\Omega$ and $H(x,y)$ is its regular part.\\
Now,
 we search for a solution $\mf u:= (u_1,\dots,u_m)$  to \eqref{pb 1} as
\begin{equation}\label{ans}
\mf u=\mf W+\bs \phi,\ \hbox{where}\ \mf W:=\left(\mathfrak c_1 PU_{\delta_1,\xi_1},\dots,\mathfrak c_qPU_{\delta_q,\xi_q}\right)\in (H_0^1(\Omega  ))^{|I_{1}|}\times\dots\times (H_0^1(\Omega  ))^{|I_{q}|},\end{equation}
where each vector $\mathfrak c_h\in\mathbb R^{|I_h|}$ is defined in \eqref{sist2intro}, 
  the concentration parameters $\d_h=e^{-{d_h\over\lambda^*_h}}$ with $ {\lambda}^*_h:=\max_{i\in I_h}\lambda_i$,  and the concentration points $\xi_h\in \Omega$ 
are such that $(\mf{d},\bs{\xi}) = (d_1,\dots,d_q, \xi_1,\dots,\xi_q)\in X_{\eta}$, with
\begin{equation}\label{def X}
X_{\eta}:=\left\{ (\mf{d},\bs{\xi}) \in \R^q \times \Omega^q: \ \eta < d_h< \eta^{-1},\ \text{dist}(\xi_h,\partial\Omega)\ge\eta, \ |\xi_h-\xi_k|\ge\eta\ \ \hbox{if}\ h\not=k\right\},
\end{equation}
for some  $\eta \in (0,1).$   Recall that $|I_1|+\dots+|I_q|=m.$
The higher order term   $ \bs \phi=(\phi_1,\dots,\phi_m) \in \left(H_0^1(\Omega  )\right)^m$ belongs to the space  $\bs{K}  ^{\perp}$  whose definition involves the solutions of  the linear equation
\begin{equation}\label{linear eq}
-\Delta \psi = 3 U_{\d,\x}^{2} \psi \quad \text{in $\R^4$}, \quad \psi \in \mathcal{D}^{1,2}(\R^4).
\end{equation}
More precisely, we know that the set of solutions to \eqref{linear eq} is a $5-$dimensional space, which is generated by (see \cite{BianchiEgnell})  
$$\begin{aligned}
\psi_{\d,\x}^0 &:= \frac{\pa U_{\d,\x}}{\pa \d} = \alpha  \frac{|x-\xi|^2 -\delta^2}{\left( \d^2 + |x-\x|^2 \right)^{2}} \\
\psi_{\d,\x}^\ell&:=   \frac{\pa U_{\d,\x}}{\pa \x_\ell} = 2\alpha_N   \d \frac{x_\ell-\xi_\ell}{\left( \d^2 + |x-\x|^2 \right)^{ {2}}}, \quad \ell=1,\dots,4.
\end{aligned}
$$
It is necessary to introduce   the projections $P  \psi_{\d,\xi}^\ell$ of $\psi_{\d,\xi}^\ell$ ($\ell=0,\dots,N$) into $H_0^1(\Omega)$, i.e.
\begin{equation}
\label{Ppsi}
 -\Delta (P\psi_{\d,\xi}^\ell) = -\Delta \psi_{\d,\x}^\ell = 3 U_{\d,\x}^{2} \psi_{\d,\x}^\ell \hbox{ in}\ \Omega,\quad
P \psi_{\d,\x}^\ell = 0\ \hbox{ on}\ \partial\Omega,
\end{equation}
and it is useful to recall that  
$$\begin{aligned}P\psi_{\delta,\xi}^0(x)&=\psi^0_{\delta,\xi}(x)-\alpha  H(x,\xi)+\mathcal O\left(\delta^{2}\right),
\\ P\psi_{\delta,\xi}^\ell(x)&=\psi^\ell_{\delta,\xi}(x)-\alpha\delta \pa_\ell H(x,\xi)+\mathcal O\left(\delta^{3}\right), \quad \ell=1,\dots,4.
\end{aligned}
$$
Now, we define   the space  $\bs{K}  ^{\perp}$   as 
	\begin{equation}
	\label{def_K}
	 \mf{K}  := K_1 \times \cdots \times K_q\
\hbox{ and }\  \bs{K}  ^{\perp} = K_1^\perp \times \cdots \times K_q^\perp,
	\end{equation}
where (see \eqref{non-de})
\begin{equation}\label{kh}
K_h:= \textrm{span}\left\{\mathfrak e_h P \psi_{\d_h,\xi_h}^\ell: \ \ell=0,\dots,4\right\}\subset \left(H_0^1(\Omega)\right)^{|I_{h}|}, \quad h=1,\dots,q.
\end{equation}
The unknowns in \eqref{ans} are the rates of the concentration parameters $d_h$'s, the concentration points $\xi_h$'s and the remainder terms $\phi_i$'s. To identify them, we will use a Ljapunov--Schmidt reduction method.
First, we rewrite system  \eqref{pb 1}   as a couple of systems.
Let  us introduce the orthogonal projections
$$\bs{\Pi}:=(\Pi_1,\dots,\Pi_q):(H_0^1(\Omega  ))^{|I_{1}|}\times\dots\times H_0^1(\Omega  ))^{|I_{q}|}  \to \mf{K}$$
and
$$
\ \bs{\Pi}^\perp:=(\Pi_1^\perp,\dots,\Pi_q^\perp):(H_0^1(\Omega  ))^{|I_{1}|}\times\dots\times H_0^1(\Omega  ))^{|I_{q}|} \to \mf{K}  ^\perp,$$
where
 $\Pi_h :\left(H_0^1(\Omega)\right)^{|I_{h}|}\to K_h$ and  $\Pi_h^{\perp} :\left(H_0^1(\Omega)\right)^{|I_{h}|} \to K_h^{\perp}$ denote  the orthogonal projections, for every $h=1,\dots,q$.
\\
It is not difficult to check that \eqref{pb 1} is equivalent to the couple of systems
\begin{equation}\label{csys1}
 \bs{\Pi}^\perp\left[\bs{\mathcal L}(\bs\phi)+\bs{\mathcal N}(\bs\phi)+\bs{\mathcal E}\right]=0\end{equation}
and \begin{equation}\label{csys2}
 \bs{\Pi}  \left[\bs{\mathcal L}(\bs\phi)+\bs{\mathcal N}(\bs\phi)+\bs{\mathcal E}\right]=0,
\end{equation}
where
  the linear operator ${\bs{\mathcal L}   (\bs{\phi})}=(\mathcal L^1  (\bs{\phi}),\dots,\mathcal L^m  (\bs{\phi})):\left(H^1_0(\Omega)\right)^m\to\left(H^1_0(\Omega)\right)^m$ is defined for every $i\in I_h$ and $h=1,\dots,q$ as
\begin{equation}
\label{def_L}
\begin{aligned} \mathcal L^i  (\bs{\phi}):=  & \phi_i -(-\Delta)^{-1} \left\{\left[\left(3\beta_{ii}{c _i}^2+\sum_{j \in I_h\atop j\neq i} \beta_{ij} {c _j}^{2}\right) \phi_i+2\sum_{j \in I_h\atop j\neq i} \beta_{ij}  c _j  c _i\phi_j\right](P U_{\d_h,\x_h})^2\right.  \\ & \left. 
\qquad+\sum_{k\not=h}\sum_{j  \in I_k }\beta_{ij} \left((c _j P  U_{\d_k,\x_hk} )^2\phi_i +2c _i c _j P  U_{\d_h,\x_h} P  U_{\d_k,\x_hk}\phi_j\right)+\lambda_i\phi_i \right\} ,\end{aligned}  
\end{equation}
the nonlinear term ${\bs{\mathcal N}  (\bs{\phi})}=(\mathcal N^1  (\bs{\phi}),\dots,\mathcal N^m  (\bs{\phi}))\in \left(L^\frac43(\Omega)\right)^m$ is defined for every $i\in I_h$ and $h=1,\dots,q$ as
\begin{equation}
\label{def_N}
\begin{aligned} \mathcal  N^i  (\bs{\phi}):=&  - (-\Delta)^{-1}\left\{\sum_{j \in I_h\atop j\neq i} \beta_{ij} (c _i P  U_{\d_h,\x_h}  \phi_j ^{2}+2c _j P  U_{\d_h,\x_h}\phi_j\phi_i+\phi_j^2\phi_i) \right.\\ & \left.
+\sum_{k\not=h}\sum_{j  \in I_k }\beta_{ij} (c _i P  U_{\d_h,\x_h}  \phi_j ^{2}+2c _j P  U_{\d_k,\x_k}\phi_j\phi_i+\phi_j^2\phi_i) 
  +\beta_{ii}\left(3 c _i P  U_{\d_h,\x_h} \phi_i^2 + \phi_i^3\right)\right\},
\end{aligned}
\end{equation}
and the error term $\bs{\mathcal E}=(\mathcal E^1   ,\dots,\mathcal E^m  )\in \left(L^\frac43(\Omega)\right)^m$ is defined for every $i\in I_h$ and $h=1,\dots,q$ as
\begin{equation}
\label{def_E}
\begin{aligned} \mathcal E^i   :=& -(-\Delta)^{-1}\left\{ \left( \sum_{j \in I_h } \beta_{ij} {c _j}^{2} c _i \right)   \left[ (P U_{\d_h,\x_h})^3-(  U_{\d_h,\x_h})^3 \right]\right.  \\ & \left. 
+ \sum_{k\not=h}\sum_{j  \in I_k } \beta_{ij}  (c _j P  U_{\d_k,\x_k} )^{2}   ( c _i P U_{\d_h,\x_h} )+\lambda_i  c _i P U_{\d_h,\x_h} 
\right\} .\end{aligned}
\end{equation}
In the above computation, we used \eqref{sist2intro} and  \eqref{sist1intro}, so that  for every $i\in I_h$ and $h=1,\dots,q$  
$$ c _i P U_{\d_h,\x_h}=(-\Delta)^{-1}\left[\left( 
 \sum_{j \in I_h } \beta_{ij} {c _j}^2c _i  \right)(   U_{\d_h,\x_h} )^3\right].$$
The proof of our main result consists of two main steps. First, for fixed $\mf d=(d_1,\dots,d_q)$, and $\bs \xi=(\xi_1,\dots,\xi_q)$ we solve  the  system \eqref{csys1}, finding $\bs{\phi}=\bs{\phi}(\mf d,\bs \xi)\in \mf{K}^\perp$. Plugging this choice of $\bs{\phi}$ into the second system   \eqref{csys2}, we obtain a finite dimensional problem in the unknowns $\mf{d}$ and $\bs{\xi}$, whose solution is identified as a critical point of a suitable function.

\section{Proof of Theorem \ref{main}}\label{due}
We briefly sketch the main steps of the proof.

\subsection*{The linear theory}

As a first step, it is important to understand the solvability of the linear problem naturally associated to    \eqref{csys1}, i.e.  given $\bs{\mathcal L}$ as in \eqref{def_L}
$$
\bs{\mathcal L}(\bs{\phi}) = \bf h, \quad \text{with} \quad \bf h \in K ^\perp.
$$
\begin{proposition}\label{linear}
 For every $\eta>0$ small enough there exist $\bar \beta>0$, $\lambda_0>0$ and $C>0$, such that if $\lambda_i \in (0,\lambda_0)$, for every $i=1,\dots,m$, and
\begin{equation}\label{beta}
 \beta^*:= \max_{(i,j)\in I_h\times I_k \atop h\not=k}\beta_{ij}\leq \overline{\beta}\,,
\end{equation}
then
\begin{equation}\label{17ott1}
\|\bs{\mathcal L}  (\bs{\phi})\|_{(H_0^1(\Omega))^m} \ge C \|\bs{\phi}\|_{(H_0^1(\Omega))^m}  \qquad \forall \bs{\phi} \in \bf K ^\perp\,,
\end{equation}
for every $(\mf{d}, \bs{\xi}) \in X_\eta$. Moreover, $\bs{\mathcal L}   $ is invertible in $\mf{K}  ^\perp$ with continuous inverse. \\
\end{proposition}
\begin{proof} It is postponed to Appendix.\end{proof}

\subsection*{The error term} We need to estimate the error term $\bs{\mathcal E}$ defined in \eqref{def_E}.
\begin{lemma}\label{errore}
For every $\eta>0$   small enough there exist $\lambda_0>0$ and $C>0$ such that, if $\lambda_i\in(0,\lambda_0)$ for every $i=1,\dots,m$, then
\begin{equation}\label{resto1}
\|\bs{\mathcal E} \|_{(H_0^1(\Omega))^m} \le C  \sum_{h=1}^q\left(O(\delta_{h}^2)+ O(\lambda^*_{h}\delta_{h})+\sum_{k\neq h}O(|\beta^*|\delta_{h}\delta_{k})\right)
\end{equation}
for every $(\mf{d}, \bs{\xi}) \in X_\eta$, where $\lambda^*_{h}:=\max\limits_{i\in I_h}\lambda_i$  and $\beta^*:=\max\limits_{(i,j)\in I_h\times I_k\atop i\not=k}\beta_{ij}.$
\end{lemma}

\begin{proof}
We argue as in \cite[Lemma A.1--A.3]{PistoiaTavares}. Note first that, by the continuity of $(-\Delta)^{-1}$, for every $i\in I_h$
\[
\begin{split}
\|\mathcal E  ^i\|\leq & C\left(  \sum_{j \in I_h } |\beta_{ij}| {c _j}^{2} c _i \right)\left|(P U_{\d_h,\x_h})^3-(  U_{\d_h,\x_h})^3\right|_{\frac{4}{3}}\\
+&C\sum_{k\not=h}\sum_{j  \in I_k } |\beta_{ij}|  c _i c _j^2\left|P U_{\d_h,\x_h}(P  U_{\d_k,\x_k} )^{2}\right|_\frac{4}{3}+\lambda_i c _i\left|P U_{\d_h,\x_h} \right|_\frac{4}{3}.
\end{split}
\]
Moreover,
\[
\left|(P U_{\d_h,\x_h})^3-(  U_{\d_h,\x_h})^3\right|_{\frac{4}{3}} =O(\delta_{h}^2),
\]
$$\left|P U_{\d_h,\x_h}P  U_{\d_k,\x_k} ^{2}\right|_\frac{4}{3}= O(\delta_h\delta_k)$$
and
\[
\left|P U_{\d_h,\x_h} \right|_\frac{4}{3}=O(\delta_{h})\,.
\]
Then the claim follows.\end{proof}

\subsection*{Solving (\ref{csys1}).}
We combine all the previous results and a standard contraction mapping argument and we prove the solvability of the system \eqref{csys1}. 
\begin{proposition}\label{phi}
 For every $\eta>0$ small enough there exist $\bar \beta>0$, $ \lambda_0>0$ and $C>0$ such that, if $\lambda_i \in (0,\lambda_0)$ for every $i=1,\dots,m$
 and \eqref{beta} holds,
then for every $(\mf{d},\bs{\xi}) \in X_\eta$ there exists a unique function $\bs{\phi}  =\bs{\phi}( \mf{d},\bs{\xi}) \in \bs{K}  ^\perp$ solving   system \eqref{csys1}. Moreover,
\begin{equation}\label{rid2}
\|\bs{\phi}   \|_{\left(H_0^1(\Omega  )\right)^m} \le C\sum_{h=1}^q\left(O(\delta_{h}^2)+O(\lambda^*_{h}\delta_{h})+\sum_{k\neq h}O(|\beta^*|\delta_{h}\delta_{k})\right)
\end{equation}
and $(\mf{d},\bs{\xi}) \mapsto \bs{\phi}( \mf{d},\bs{\xi}) $ is a $C^1-$function.
\end{proposition}

\begin{proof}
The claim follows by Proposition \ref{linear} and Lemma \ref{errore} arguing exactly as in \cite[Proposition 3.2 and Lemma 3.3]{PistoiaTavares}, noting that the nonlinear part $\bs{\mathcal N}$ given in \eqref{def_N} has a quadratic growth in $\bs\phi$. In particular \eqref{rid2} follows by \eqref{resto1}
\end{proof}

\subsection*{The reduced problem} 
Once the first system \eqref{csys1} has been solved, we have to find a solution to the second system \eqref{csys2} and so a solution to system \eqref{system}.

\begin{proposition}\label{ridotto} For any $\eta>0$ small enough there exist $\overline\beta>0$ and $\lambda_0>0$  such that,
if $\lambda_i\in(0,\lambda_0)$ for every $i=1,\dots,m$ and \eqref{beta} holds, then
$\mf u=\mf W(\mf d,\bs\xi)+\bs\phi(\mf d,\bs\xi)$ defined in \eqref{ans} solves system \eqref{system}, i.e. it is a critical point of the energy
$$J(\mf u):=\frac12\sum\limits_{i=1}^m\int\limits_{\Omega} |\nabla u_i|^2- \frac14\sum\limits_{i,j=1 }^m\beta_{ij} \int\limits_{\Omega}u_i^2u_j^2-\frac12\sum\limits_{i=1}^m\int\limits_{\Omega}\lambda_iu_i^2 $$
if and only if $(\mf d,\bs\xi)\in X_\eta$   is a critical point of the reduced energy
$$\tilde J(\mf d,\bs\xi):=J\left(\mf W+\bs \phi\right).$$
Moreover, the following expansion holds true
\begin{equation}\label{en-rid}\tilde J(\bs\delta,\bs\xi)=\sum\limits_{h=1}^q\left(\sum\limits_{i\in I_h}c _i^2\right)\left( A_0+  A_1\delta_h^2\mathtt r(\xi_h)+   A_2\lambda_h^*\delta_h^2|\ln\delta_h| +o(\delta_h^2)\right)\end{equation}
$C^1-$uniformly   in $X_\eta.$
Here the $ A_i$'s are positive constants, $\mathtt r$  is the Robin's function and $\lambda_h^*=\max\limits_{i\in I_h}\lambda_i.$
\end{proposition}
\begin{proof}
The proof follows by combining the arguments in  \cite[Section 3 and Section 5]{PistoiaTavares} and  \cite[Section 5]{PistoiaSoave}. We   remark that in this case the fact that  $\bs\phi(\mf d,\bs\xi)$ solves \eqref{csys1} is equivalent to claim that it solves the system
$$ \bs{\mathcal L}(\bs\phi) -\bs{\mathcal N}(\bs\phi)- \bs{\mathcal E}=\left(\sum_{\ell=0}^4a_1^\ell\mathfrak e_1 P \psi_{\d_1,\xi_1}^\ell,\dots,\sum_{\ell=0}^4a_h^\ell\mathfrak e_h P \psi_{\d_h,\xi_h}^\ell\right),
$$
for some real  numbers $a_i^\ell.$ Therefore, $\mf W(\mf d,\bs\xi)+\bs\phi(\mf d,\bs\xi)$ solves system \eqref{csys2} if and only if all the $a_i^\ell$'s are zero.
We also point out that it is quite standard to prove that
$J\left(\mf W+\bs \phi\right)\approx J\left(\mf W \right)$
and moreover by \eqref{sist2intro} we deduce
$$\begin{aligned}J\left(\mf W \right)&=\sum\limits_{h=1}^q \underbrace{\left(\sum\limits_{i\in I_h}\frac12 c_i^2\int\limits_\Omega| \nabla PU_{\delta_h,\xi_h}|^2-\frac14 \sum\limits_{i,j\in I_h }\beta_{ij} (c_i c_j)^2\int\limits_\Omega(PU_{\delta_h,\xi_h})^4 \right) }_{=\left(\sum\limits_{i\in I_h} c_i^2\right)\left(\frac12\int\limits_\Omega| \nabla PU_{\delta_h,\xi_h}|^2-\frac14  \int\limits_\Omega(PU_{\delta_h,\xi_h})^4\right)}\\
&-\frac12\sum\limits_{h,k=1\atop h\not=k}\beta_{ij}\int\limits_\Omega ( c_iPU_{\delta_h,\xi_h})^2( c_j PU_{\delta_k,\xi_k})^2-
\sum\limits_{h=1}^q\sum\limits_{i\in I_h}\frac12\lambda_i \int\limits_\Omega ( c_iPU_{\delta_h,\xi_h})^2
\end{aligned}$$
so that the claim follows just arguing as in \cite{PistoiaTavares}.\\
\end{proof}
\subsection*{Proof of Theorem \ref{main}: completed}
Arguing exactly as in \cite[Proof of Theorem 1.3, p. 437]{PistoiaTavares}, we prove that the reduced energy \eqref{en-rid} has a critical point $(\mf d_{\bs\lambda},\bs\xi_{\bs\lambda})$ provided $\bs\lambda= (\lambda_1,\dots,\lambda_m)$ is small enough and    $\bs \xi_{\bs\lambda}\to (\xi^0_1,\dots,\xi_q^0)$ as $\lambda^*=\max_i\lambda_i\to0.$ Theorem \ref{main} immediately follows by   Proposition \ref{ridotto}.
Moreover, if  the $\beta_{ij}$'s depend on the $\lambda_i$'s  and $\beta^*$  satisfies
 $|\beta^*|=O\left(e^{d^*\over\lambda^*}\right)$ with $d^*<\min _{h=1,\dots,q}d_i,$ then for every $h=1,\dots,m$
 $$|\beta^*|\delta_h\lesssim e^{{d^* \over\lambda^*}-{d_h \over\lambda^*_h}}\lesssim e^{{d^*  -d_h \over\lambda^*}} =o(1)$$
and by estimate \eqref{rid2} we can still conclude  the validity of \eqref{en-rid}, and so the last part of Theorem \ref{main} follows (see also \cite[Section 5.3, p.438]{PistoiaTavares}.
\section {Appendix}

\subsection* {Proof of Proposition \ref{linear}}
We argue combining ideas of \cite[Lemma 3.1]{PistoiaTavares} and \cite[Lemma 5.4]{PistoiaSoave}. We first prove \eqref{17ott1} by contradiction. Assume thus that there exist $\{({\bf d}_n,\boldsymbol{\xi}_n)\}_n\subset X_\eta$ so that $\boldsymbol{\xi}_n\to\boldsymbol{\xi}$ as $n\to+\infty$, $\boldsymbol{\lambda}_n:=(\lambda_{1,n},\,\dots,\,\lambda_{m,n})\to0$ as $n\to+\infty$, and $\boldsymbol{\phi}^n:=(\phi_1^n,\,\dots,\,\phi_m^n)\in{\bf K}^\perp$ so that $\|\boldsymbol{\phi}^n\|=1$ for every $n\in\mathbb{N}$ and
\[
\|\bs{\mathcal L} (\boldsymbol{\phi}^n)\|\to 0\qquad\text{as }n\to+\infty\,.
\]
We recall that the spaces introduced in \eqref{def_K} and \eqref{kh} depend on $\bf d_n$ and $\bs \xi_n,$ so
for the sake of clarity,  let us introduce the following notation. For every $h=1,\,\dots,\,q$, let 
\[
\begin{split}
K_h^n&:=K_{d_{h,n},\xi_{h,n} },\quad\quad (K_h^n)^\perp:=K_{d_{h,n},\xi_{h,n} }^\perp,\\
U_h^n&:=U_{\delta_{h,n},\xi_{h,n}},\qquad\quad PU_h^n:=PU_{\delta_{h,n},\xi_{h,n}},\\
\psi_{h,n}^l&:=\psi_{\delta_{h,n},\xi_{h,n}}^l,\quad\quad\,\, P\psi_{h,n}^l:=P\psi_{\delta_{h,n},\xi_{h,n}}^l,\qquad l=0,\,\dots,\,4\,,
\end{split}
\] 
where $\delta_{h,n}:=e^{-\frac{d_{h,n}}{ {\lambda}^*_{h,n}}}$ and $ {\lambda}^*_{h,n}:=\max_{i\in I_h}\lambda_{i,n}$. Moreover, set ${\bf h}_n:=\bs{\mathcal  L} (\boldsymbol{\phi}^n)$.

By definition of $\bs{\mathcal  L}$ and the fact that $\boldsymbol{\phi}^n\in {\bf K} ^\perp$, we have, for every $h=1,\,\dots,\, q$ and $i\in I_h$ 
\begin{equation}
\label{phi_in_perp}
\begin{split}
\phi_i^n=&(-\Delta)^{-1}\left\{\left[\left(3\mu_i  c _i^2+\sum_{j\in I_h\atop j\neq i}\beta_{ij} c _j^2\right)\phi_i^n+2\sum_{j \in I_h\atop j\neq i}\beta_{ij} c _i c _j\phi_j^n\right]\left(PU_h^n\right)^2\right.\\
&\left.+\sum_{k\neq h}\sum_{j \in I_k}\beta_{ij}\left[ c _j^2(PU_k^n)^2\phi_i^n+2 c _i c _jPU_k^n PU_h^n\phi_j^n\right]+\lambda_{i,n}\phi_i^n\right\}+h_i^n-w_i^n\,,
\end{split}
\end{equation}
for a suitable ${\bf w}_n:=(w_i^n)_{i\in I_h}\in K_h^n$. Here   $\mu_i:=\beta_{ii}.$

\smallskip
{\em Step 1: $\|{\bf w}_n\|\to0$ as $n\to+\infty$.} Multiplying \eqref{phi_in_perp} by $\delta_{h,n}^2w_i^n$ and recalling the definition of $(-\Delta)^{-1}$ yields
\[
\begin{split}
\delta_{h,n}^2\|w_i^n\|^2=&\delta_{h,n}^2\langle h_i^n-\phi_i^n,w_i^n\rangle +\delta_{h,n}^2\left(3\mu_i c _i^2+\sum_{j \in I_h\atop j\neq i}\beta_{ij} c _j^2\right)\int (PU_h^n)^2\phi_i^n w_i^n\\
+&2\delta_{h,n}^2\sum_{j \in I_h\atop j\neq i}\beta_{ij} c _i c _j\int(PU_h^n)^2\phi_j^n w_i^n+\delta_{h,n}^2\sum_{k\neq h}\sum_{j  \in I_k }\beta_{ij} c _j^2\int(PU_k^n)^2\phi_i^n w_i^n\\
+&2\delta_{h,n}^2 c _i\sum_{k\neq h}\sum_{j  \in I_k }\beta_{ij} c _j\int PU_k^n PU_h^n\phi_j^n w_i^n +\delta_{h,n}^2\lambda_{i,n}\int \phi_i^n w_i^n\,,
\end{split}
\]
so that, summing over $i\in I_h$ and making use of $(\phi_i^n)_{i\in I_h},(h_i^n)_{i\in I_h}\in (K_h^n)^\perp$,
\begin{equation}
	\label{norm_w1}
	\begin{split}
		&\underbrace{\delta_{h,n}^2\sum_{i\in I_h}\|w_i^n\|^2}_{I}\\=&\delta_{h,n}^2\underbrace{\sum_{i\in I_h}\left[\left(3\mu_i c _i^2+\sum_{j \in I_h\atop j\neq i}\beta_{ij} c _j^2\right)\int (PU_h^n)^2\phi_i^n w_i^n+2\delta_{h,n}^2\sum_{j \in I_h\atop j\neq i}\beta_{ij} c _i c _j\int(PU_h^n)^2\phi_j^n w_i^n\right]}_{II}\\
		+&\delta_{h,n}^2\underbrace{\sum_{i\in I_h}\sum_{k\neq h}\sum_{j  \in I_k }\beta_{ij} c _j^2\int(PU_k^n)^2\phi_i^n w_i^n}_{III}+2\delta_{h,n}^2\underbrace{\sum_{i\in I_h} c _i\sum_{k\neq h}\sum_{j  \in I_k }\beta_{ij} c _j\int PU_k^n PU_h^n\phi_j^n w_i^n}_{IV}\\
		+&\delta_{h,n}^2\underbrace{\sum_{i\in I_h}\lambda_{i,n}\int \phi_i^n w_i^n}_{V}\,.
	\end{split}
\end{equation}
Note first that, since $(w_i^n)_{i\in I_h}\in K_h^n$, for $l=0,\,\dots,\,4$ there are $a_{h,n}^l\in\mathbb{R}$ for which it holds (see \eqref{def_K})
\[
(w_i^n)_{i\in I_h}=\sum_{l=0}^n a_{h,n}^l\mathfrak{e}_h P\psi_{h,n}^l\,,
\]
so that arguing as in \cite[p. 417]{PistoiaTavares} and for sufficiently large $n$ we can write
\begin{equation}
\label{I}
I=\delta_{h,n}^2\sum_{i\in I_h}\sum_{l,p=0}^na_{h,n}^la_{h,n}^p|\mathfrak{e}_{i,h}|^2\int\nabla P\psi_{h,n}^l\cdot\nabla P\psi_{h,n}^p= \sum_{l=0}^n(a_{h,n}^l)^2\sigma_{ll}+o(1)\sum_{l=0}^na_{h,n}^l a_{h,n}^p\,,
\end{equation}
for suitable positive constants $\sigma_{ll}$, $l=0,\,\dots,\,4$.

Let us thus estimate terms $III$ and $IV$ in \eqref{norm_w1}. On the one hand, for every $h,k=1,\,\dots, q$, $k\neq h$, $i\in I_h$ and $l=0,\,\dots,4$,
\begin{equation}
\label{estIII_1}
\begin{split}
\left|\int (PU_k^n)^2\phi_i^n P\psi_{h,n}^l\right|\leq& \left|\int(PU_k^n)^2\phi_i^n\psi_{h,n}^l\right|+\left|\int(PU_k^n)^2\phi_i^n(P\psi_{h,n}^l-\psi_{h,n}^l)\right|\\
\leq&\|\phi_i^n\|\left|(PU_n^k)^2\psi_{h,n}^l\right|_{\frac{4}{3}}+\|\phi_i^n\|\left|(PU_k^n)^2(P\psi_{h,n}^l-\psi_{h,n}^l)\right|_{\frac{4}{3}}\,,
\end{split}
\end{equation}
where we made use of H\"older and Sobolev inequality. Then, by \cite[Lemma A.1]{PistoiaTavares} we get
\begin{equation}
\label{l=0}
\left|(PU_k^n)^2(P\psi_{h,n}^0-\psi_{h,n}^0)\right|_{\frac{4}{3}}\leq C(\delta_{h,n}+o(\delta_{h,n})\left(\int(PU_k^n)^{\frac{8}{3}}H(\cdot,\xi_{k,n})^{\frac{4}{3}}\right)^{\frac{3}{4}}\leq C'\delta_{h,n}+o(\delta_{h,n})
\end{equation}
and
\begin{equation}
\label{l=1,...,4}
\left|(PU_k^n)^2(P\psi_{h,n}^l-\psi_{h,n}^l)\right|_{\frac{4}{3}}\leq C(\delta_{h,n}^2+o(\delta_{h,n}^2))\left(\int(PU_k^n)^{\frac{8}{3}}\frac{\partial H}{\partial\xi}(\cdot,\xi_{k,n})^{\frac{4}{3}}\right)\leq C'\delta_{h,n}^2+o(\delta_{h,n}^2)
\end{equation}
for every $l=1,\,\dots,\,4$. Moreover, since direct calculations show
\[
\begin{split}
&|\psi_{\delta,\xi}^0|\leq\frac{C}{\delta}U_{\delta,\xi}\\
&|\psi_{\delta,\xi}^l|\leq\frac{C}{\delta}U_{\delta,\xi}^2|x_l-\xi_l|,\quad l=1,\,\dots,\,4\,,
\end{split}
\]
recalling that $0\leq PU_k^n\leq U_k^n$ by the maximum principle and making use of \cite[Lemma A.2--A.4]{PistoiaTavares}, we also have
\[
\begin{split}
\left|(PU_k^n)\psi_{h,n}^0\right|_{\frac{4}{3}}\leq&\frac{C}{\delta_{h,n}}\left|(PU_k^n)^2 U_h^n\right|_{\frac{4}{3}}\leq\frac{C}{\delta_{h,n}}\left|(U_k^n)^2 U_h^n\right|_{\frac{4}{3}}\\
\leq&\frac{C'}{\delta_{h,n}}\left(O(\delta_{h,n}\delta_{k,n})+O(\delta_{k,n}^2\delta_{h,n})\right)=O(\delta_{k,n})
\end{split}
\]
and, for $l=1,\,\dots,\,4$, 
\[
\begin{split}
\left|(PU_k^n)\psi_{h,n}^l\right|_{\frac{4}{3}}\leq\frac{C}{\delta_{h,n}}\left|(U_k^n)^2(U_h^n)^2\right|_{\frac{4}{3}}\leq\frac{C'}{\delta_{h,n}}\left(O(\delta_{k,n}^2\delta_{h,n})+O(\delta_{k,n}\delta_{h,n}^2)\right)=o(\delta_{k,n})\,.
\end{split}
\]
Combining the previous estimates with \eqref{estIII_1} and the fact that $\|\boldsymbol{\phi}_n\|=1$ thus leads to 
\begin{equation}
	\label{III}
	\delta_{h,n}^2 |III|\leq o(\delta_{h,n}^2)\sum_{l=0}^n |a_{h,n}^l|\,. 
\end{equation}
Similarly, by H\"older and Sobolev inequality, $\|\boldsymbol{\phi}_n\|=1$ and \cite[Lemma A.5]{PistoiaTavares}, 
\[
\left|\int PU_k^n PU_h^n \phi_j^n P\psi_{h,n}^l\right|\leq C\left|PU_k^n PU_h^n P\psi_{h,n}^l\right|_{\frac{4}{3}}=O(\delta_{k,n}\delta_{h,n})\,,
\]
for every $k,h=1,\,\dots,\,q$, $k\neq h$, $j\in I_k$ and $l=0,\,\dots,\,4$, so that
\begin{equation}
	\label{est IV}
	\delta_{h,n}^2|IV|\leq o(\delta_{h,n}^2)\sum_{l=0}^n |a_{h,n}^l|\,.
\end{equation} 
Furthermore, by H\"older inequality and recalling that $\boldsymbol{\lambda}_n\to0$ as $n\to+\infty$ and $\|\boldsymbol{\phi}_n\|=1$, we also have
\begin{equation}
	\label{est V}
	\delta_{h,n}^2|V|\leq o(\delta_{h,n}^2)\|(w_i^n)_{i\in I_h}\|\,.
\end{equation}
We are thus left to estimate the term $II$ in \eqref{norm_w1}. To this purpose, we set, for every $i,j\in I_h$,
\[
\alpha_{ii}:=3\mu_i c _i^2+\sum_{j \in I_h\atop j\neq i}\beta_{ij} c _j^2,\qquad\alpha_{ij}:=2\beta_{ij} c _i c _j\,,
\] 
so that
\[
\begin{split}
II=&\sum_{i\in I_h}\int\left(\alpha_{ii}\phi_i^n+\sum_{j \in I_h\atop j\neq i}\alpha_{ij}\phi_j^n\right)(PU_h^n)^2 w_i^n\\
=&\underbrace{\sum_{i\in I_h}\int\left(\alpha_{ii}\phi_i^n+\sum_{j \in I_h\atop j\neq i}\alpha_{ij}\phi_j^n\right)\left((PU_h^n)^2-(U_h^n)^2\right)w_i^n}_{II.1}\\+&\underbrace{\sum_{i\in I_h}\int\left(\alpha_{ii}\phi_i^n+\sum_{j \in I_h\atop j\neq i}\alpha_{ij}\phi_j^n\right)(U_h^n)^2 w_i^n}_{II.2}\,.
\end{split}
\]
As for $II.1$, we have, for every $i,j\in I_h$
\[
\left|\int \left((PU_h^n)^2-(U_h^n)^2\right)\phi_j^n w_i^n\right|\leq\left|(PU_h^n)^2-(U_h^n)^2\right|_2|\phi_j^n|_4|w_i^n|_4\leq \left|(PU_h^n)^2-(U_h^n)^2\right|_2\|w_i^n\|
\]
by H\"older and Sobolev inequality and $\|\boldsymbol{\phi}_n\|=1$. Furthermore, by \cite[Lemma A.3]{PistoiaTavares} and $0\leq PU_h^n\leq U_h^n$
\[
\left|(PU_h^n)^2-(U_h^n)^2\right|_2\leq C\left(\left|U_h^n(PU_h^n-U_h^n)\right|_2+\left|(PU_h^n-U_h^n)^2\right|_2\right)\leq C'\left|U_h^n(PU_h^n-U_h^n)\right|_2,
\]
and by \cite[Lemma A.1--A.2]{PistoiaTavares}
\[
\begin{split}
\left|U_h^n(PU_h^n-U_h^n)\right|_2=&\left(\int (U_h^n)^2(PU_h^n-U_h^n)^2\right)^{\frac{1}{2}}=\left(\int (U_h^n)^2(\delta_{h,n}AH(\cdot,\xi_{h,n})+o(\delta_{h,n}))^2\right)^{\frac{1}{2}}\\
\leq&(C\delta_{h,n}+o(\delta_{h,n}))\left(\int (U_h^n)^2\right)^\frac{1}{2}\leq C'\delta_{h,n}^2|\ln\delta_{h,n}|^{\frac{1}{2}}+o(\delta_{h,n}^2)\,,
\end{split}
\]
in turn yielding
\begin{equation}
\label{II.1}
|II.1|\leq C\sum_{i,j\in I_h}\left|\int \left((PU_h^n)^2-(U_h^n)^2\right)\phi_j^n w_i^n\right|\leq C'(\delta_{h,n}^2|\ln\delta_{h,n}|+o(\delta_{h,n}^2))\|(w_i^n)_{i\in I_h}\|\,.
\end{equation}
To estimate $II.2$, note first that
\begin{equation}
\label{prelII.2}
\begin{split}
II.2=&\sum_{i\in I_h}\int\left(\alpha_{ii}\phi_i^n+\sum_{j \in I_h\atop j\neq i}\alpha_{ij}\phi_j^n\right)(U_h^n)^2\sum_{l=0}^n a_{h,n}^l\mathfrak{e}_{i,h}P\psi_{h,n}^l\\
=&\sum_{l=0}^n a_{h,n}^l\int (U_h^n)^2P\psi_{h,n}^l\sum_{i\in I_h}\left(\alpha_{ii}\mathfrak{e}_{i,h}\phi_i^n+\mathfrak{e}_{i,h}\sum_{j \in I_h\atop j\neq i}\alpha_{ij}\phi_j^n\right)\\
=&\sum_{l=0}^n a_{h,n}^l\int (U_h^n)^2P\psi_{h,n}^l\sum_{i\in I_h}\phi_i^n\left(\alpha_{ii}\mathfrak{e}_{i,h}+\sum_{j \in I_h\atop j\neq i}\alpha_{ij}\mathfrak{e}_{j,h}\right)\\
=&3\sum_{l=0}^n a_{h,n}^l \int (U_h^n)^2P\psi_{h,n}^l\sum_{i\in I_h}\phi_i^n\mathfrak{e}_{i,h}\,,
\end{split}
\end{equation}
since by construction $\mathfrak{e}_h$ is an eigenvector of the matrix $(\alpha_{ij})_{i,j\in I_h}$ corresponding to the eigenvalue $\Lambda_1=3$ (see \cite[Lemma 6.1]{PistoiaSoave}). 
Recalling that $(\phi_i^n)_{i\in I_h}\in (K_h^n)^\perp$, so that by \eqref{Ppsi} and \eqref{def_K}
\[
0=\sum_{i\in I_h}\int \mathfrak{e}_{i,h}\nabla P\psi_{h,n}^l\cdot\nabla\phi_i^n=3\sum_{i\in I_h}\int (U_h^n)^2\mathfrak{e}_{i,h}\psi_{h,n}^l\phi_i^n
\]
for every $l=0,\,\dots,\,4$, we can then rewrite \eqref{prelII.2} as
\[
II.2=3\sum_{l=0}^n\sum_{i\in I_h} a_{h,n}^l \mathfrak{e}_{i,h} \int (U_h^n)^2(P\psi_{h,n}^l-\psi_{h,n}^l)\phi_i^n\,.
\]
Arguing as in \eqref{estIII_1}--\eqref{l=0}--\eqref{l=1,...,4} above we get
\[
\left|\int (U_h^n)^2(P\psi_{h,n}^l-\psi_{h,n}^l)\phi_i^n\right|\leq C\delta_{h,n}+o(\delta_{h,n})
\]
for every $l=0,\dots,4$, $h=1,\dots,q$ and $i\in I_h$, thus implying
\begin{equation}
	\label{II2}
	|II.2|\leq (C\delta_{h,n}+o(\delta_{h,n}))\sum_{l=0}^n|a_{h,n}^l|\,.
\end{equation}
Coupling \eqref{II.1}--\eqref{II2} then gives
\[
|II|\leq C'(\delta_{h,n}^2|\ln\delta_{h,n}|+o(\delta_{h,n}^2))\|{\bf w}_n\|+(C\delta_{h,n}+o(\delta_{h,n}))\sum_{l=0}^n |a_{h,n}^l|\,,
\]
and combining with \eqref{norm_w1},\eqref{III},\eqref{est IV},\eqref{est V} we finally obtain
\[
\delta_{h,n}^2\|(w_i^n)\|_{i\in I_h}^2\leq o(\delta_{h,n}^2)\|(w_i^n)\|_{i\in I_h}+o(\delta_{h,n}^2)\sum_{l=0}^n|a_{h,n}^l|\,.
\]
Together with \eqref{I}, this ensures that $\|(w_i^n)\|_{i\in I_h}\|\to0$ as $n\to+\infty$, and repeating the argument for every $h=1,\dots,q$, gives $\|{\bf w}_n\|\to0$ as desired.

\smallskip
{\em Step 2.} For every $h=1,\dots,q$ and $i\in I_h$, we set
\[
\widetilde{\phi}_i^n(y):=\begin{cases}
\delta_{h,n}\phi_i^n\left(\xi_{h,n}+\delta_{h,n}y\right) & \text{if }y\in\widetilde{\Omega}_{h,n}:=\frac{\Omega-\xi_{h,n}}{\delta_{h,n}}\\
0 & \text{if }y\in\mathbb{R}^n\setminus\widetilde{\Omega}_{h,n}\,.
\end{cases}
\]
By definition, $\|\widetilde{\phi}_i^n\|_{H_0^1(\mathbb{R}^n)}=\|\phi_i^n\|_{H_0^1(\Omega)}$, so that $\widetilde{\phi}_i^n\rightharpoonup\widetilde{\phi}_i$ in $\mathcal{D}^{1,2}(\mathbb{R}^n)$ as $n\to+\infty$, for some $\widetilde{\phi}_i$.
Let us thus show that $\widetilde{\phi}_i\equiv0$ for every $i=1,\dots, m$. To this aim, note first that \eqref{phi_in_perp} can be rewritten as
\[
\begin{split}
&\int_{\widetilde{\Omega}_{h,n}}\nabla\widetilde{\phi}_i^n\cdot\nabla\varphi\\
=&\underbrace{\delta_{h,n}^2\int_{\widetilde{\Omega}_{h,n}}\left[\left(3\mu_i  c _i^2+\sum_{j\in I_h\atop j\neq i}\beta_{ij} c _j^2\right)\widetilde{\phi}_i^n+2\sum_{j \in I_h\atop j\neq i}\beta_{ij} c _i c _j\widetilde{\phi}_j^n\right]\left(PU_h^n\right)^2(\xi_{h,n}+\delta_{h,n}y)\varphi}_{A_n}\\
+&\underbrace{\delta_{h,n}^2\int_{\widetilde{\Omega}_{h,n}}\sum_{k\neq h}\sum_{j \in I_k}\beta_{ij} c _j^2(PU_k^n)^2(\xi_{h,n}+\delta_{h,n}y)\widetilde{\phi}_i^n\varphi}_{B_n}\\
+&\underbrace{2\delta_{h,n}^3\int_{\widetilde{\Omega}_{h,n}}\sum_{k\neq h}\sum_{j \in I_k} c _i c _jPU_k^n(\xi_{h,n}+\delta_{h,n}y) PU_h^n(\xi_{h,n}+\delta_{h,n}y)\phi_j^n(\xi_{h,n}+\delta_{h,n}y)\varphi}_{C_n}\\
+&\delta_{h,n}^2\int_{\widetilde{\Omega}_{h,n}}\lambda_{i,n}\widetilde{\phi}_i^n\varphi+\int_{\widetilde{\Omega}_{h,n}}\nabla(\widetilde{h}_i^n-\widetilde{w}_i^n)\cdot\nabla\varphi,
\end{split}
\]
for every $\varphi\in C_C^\infty(\mathbb{R}^n)$, where 
\[
\begin{split}
\widetilde{h}_i^n(y):=&\begin{cases}
\delta_{h,n}h_i^n(\xi_{h,n}+\delta_{h,n}y) & \text{if }y\in\widetilde{\Omega}_{h,n}\\
0 & \text{if }y\in\mathbb{R}^n\setminus\widetilde{\Omega}_{h,n}
\end{cases}\\
\widetilde{w}_i^n(y):=&\begin{cases}
\delta_{h,n}w_i^n(\xi_{h,n}+\delta_{h,n}y) & \text{if }y\in\widetilde{\Omega}_{h,n}\\
0 & \text{if }y\in\mathbb{R}^n\setminus\widetilde{\Omega}_{h,n}\,.
\end{cases}
\end{split}
\]
Let now $\varphi$ be so that $K_\varphi:=\text{supp}\varphi\subset\widetilde{\Omega}_{h,n}$, which is always true for any given $\varphi\in C_c^\infty(\mathbb{R}^n)$ and $n$ large enough. On the one hand, it is readily seen that
\[
\begin{split}
&\delta_{h,n}^2\int_{\widetilde{\Omega}_{h,n}}\lambda_{i,n}\widetilde{\phi}_i^n\varphi\to0\\
&\int_{\widetilde{\Omega}_{h,n}}\nabla(\widetilde{h}_i^n-\widetilde{w}_i^n)\cdot\nabla\varphi\to0
\end{split}
\]
as $n\to+\infty$, since $\boldsymbol{\lambda}_n\to0$, $\|{\bf h}_n\|\to0$ and $\|{\bf w}_n\|\to0$.

On the other hand, for every $i\in I_h$, $j\in I_k$, $h\neq k$
\[
\begin{split}
\int_{\widetilde{\Omega}_{h,n}}(PU_k^n)^2(\xi_{h,n}+&\delta_{h,n}y)\widetilde{\phi}_i^n(y)\varphi(y)=\int_{\widetilde{\Omega}_{h,n}}(U_k^n)^2(\xi_{h,n}+\delta_{h,n}y)\widetilde{\phi}_i^n(y)\varphi(y)+o(1)\\
=&\delta_{k,n}^2\int_{\widetilde{\Omega}_{h,n}\cap K_\varphi}\frac{\alpha_4^2}{(\delta_{k,n}^2+|\delta_{h,n}y+\xi_{h,n}-\xi_{k,n}|^2)^2}\widetilde{\phi}_i^n(y)\varphi(y)+o(1)\to0
\end{split}
\]
and 
\[
\begin{split}
\delta_{h,n}^3\int_{\widetilde{\Omega}_{h,n}}&PU_k^n(\xi_{h,n}+\delta_{h,n}y) PU_h^n(\xi_{h,n}+\delta_{h,n}y)\phi_j^n(\xi_{h,n}+\delta_{h,n}y)\varphi(y)\\
=&\delta_{h,n}^2\int_{\widetilde{\Omega}_{h,n}\cap K_\varphi}\frac{\varphi(y)}{1+|y|^2}\frac{\alpha_4\delta_{k,n}}{\delta_{k,n}^2+|\delta_{h,n}y+\xi_{h,n}-\xi_{k,n}|^2}\phi_j^n(\xi_{h,n}+\delta_{h,n}y)+o(1)\to0
\end{split}
\]
as $y\in K_\varphi$, which is fixed and bounded, and $|\xi_{h,n}-\xi_{k,n}|\geq\eta$ by assumption. Hence, 
\[
B_n\to0,\quad C_n\to0\quad\text{as }n\to+\infty\,.
\]
Furthermore, for every $j\in I_h$, 
\[
\begin{split}
\delta_{h,n}^2\int_{\widetilde{\Omega}_{h,n}}\left(PU_h^n\right)^2(\xi_{h,n}+&\delta_{h,n}y)\widetilde{\phi}_j^n(y)\varphi(y)\\=&\delta_{h,n}^2\int_{\widetilde{\Omega}_{h,n}}\left(U_h^n\right)^2(\xi_{h,n}+\delta_{h,n}y)\widetilde{\phi}_j^n(y)\varphi(y)+o(1)\\
=&\int_{\widetilde{\Omega}_{h,n}\cap K_\varphi}\frac{\alpha_4^2}{(1+|y|^2)^2}\widetilde{\phi}_j^n(y)\varphi(y)+o(1)\to\int_{\mathbb{R}^n}\frac{\alpha_4^2}{(1+|y|^2)^2}\widetilde{\phi}_j(y)\varphi(y)\\=&\int_{\mathbb{R}^n}U_{1,0}^2(y)\widetilde{\phi}_j(y)\varphi(y)
\end{split}
\] 
since $\widetilde{\phi}_j^n\rightharpoonup\widetilde{\phi}_j$ in $\mathcal{D}^{1,2}(\mathbb{R}^n)$ and $U_{1,0}\in L^n(\mathbb{R}^n)$. Therefore, 
\[
A_n\to\int_{\mathbb{R}^n}\left[\left(3\mu_i  c _i^2+\sum_{j\in I_h\atop j\neq i}\beta_{ij} c _j^2\right)\widetilde{\phi}_i+2\sum_{j \in I_h\atop j\neq i}\beta_{ij} c _i c _j\widetilde{\phi}_j\right]\left(U_{1,0}\right)^2\varphi\qquad\text{as }n\to+\infty
\]
that is, for every $i\in I_h$
\[
-\Delta\widetilde{\phi}_i=\left[\left(3\mu_i  c _i^2+\sum_{j\in I_h\atop j\neq i}\beta_{ij} c _j^2\right)\widetilde{\phi}_i+2\sum_{j \in I_h\atop j\neq i}\beta_{ij} c _i c _j\widetilde{\phi}_j\right]\left(U_{1,0}\right)^2\quad\text{in }\mathbb{R}^n,\quad\widetilde{\phi}_i\in\mathcal{D}^{1,2}(\mathbb{R}^n).
\]
Therefore, the weak limit $(\widetilde{\phi}_i)_{i\in I_h}$ solves the linearized system \eqref{sist3intro} for every $h=1,\dots,q$. Thus, $(\widetilde{\phi}_i)_{i\in I_h}\in\text{span}\{\mathfrak{e}_h\psi_{1,0}^l\,:\,l=0,\dots,4\}$. However, since $(\phi_i^n)_{i\in I_h}\in(K_h^n)^\perp$ for every $n$, then it follows
\[
\begin{split}
0=&\delta_{h,n}\langle(\phi_i^n)_{i\in I_h},\mathfrak{e}_hP\psi_{h,n}^0\rangle=3\sum_{i\in I_h}\delta_{h,n}\int_{\Omega} (U_h^n)^2\mathfrak{e}_{i,h}\psi_{h,n}^0\phi_i^n\\
=&3\sum_{i\in I_h}\int_{\widetilde{\Omega}_{h,n}}\mathfrak{e}_{i,h}\alpha_4^3\frac{|y|^2-1}{(1+|y|^2)^n}\widetilde{\phi}_{i}^n=3\sum_{i\in I_h}\int_{\widetilde{\Omega}_{h,n}}U_{1,0}^2\psi_{1,0}^0\mathfrak{e}_{i,h}\widetilde{\phi}_i^n
\end{split}
\]
and, for every $l=1,\dots,4$,
\[
\begin{split}
0=&\delta_{h,n}\langle(\phi_i^n)_{i\in I_h},\mathfrak{e}_hP\psi_{h,n}^l\rangle=3\sum_{i\in I_h}\delta_{h,n}\int_{\Omega} (U_h^n)^2\mathfrak{e}_{i,h}\psi_{h,n}^l\phi_i^n\\
=&3\sum_{i\in I_h}\int_{\widetilde{\Omega}_{h,n}}\mathfrak{e}_{i,h}2\alpha_4^3\frac{y_l}{(1+|y|^2)^n}\widetilde{\phi}_i^n=3\sum_{i\in I_h}\int_{\widetilde{\Omega}_{h,n}}U_{1,0}^2\psi_{1,0}^l\mathfrak{e}_{i,h}\widetilde{\phi}_i^n.
\end{split}
\]
Passing to the limit as $n\to+\infty$ and making use of $\widetilde{\phi}_i^n\rightharpoonup\widetilde{\phi}_i$, we obtain
\[
3\sum_{i\in I_h}\int_{\mathbb{R}^n}U_{1,0}^2\psi_{1,0}^l\mathfrak{e}_{i,h}\widetilde{\phi}_i=0,\qquad l=0,\dots,4\,.
\]
This shows that $(\widetilde{\phi}_i)_{i\in I_h}\in(\text{span}\{\mathfrak{e}_h\psi_{1,0}^l\,:\,l=0,\dots,4\})^\perp$, thus implying $\widetilde{\phi}_i\equiv0$ for every $i\in I_h$ and concluding Step 2.

\smallskip
{\em Step 3. } We now prove that $\phi_{i}^n\to0$ strongly in $H_0^1(\Omega)$ for every $i=1,\dots,m$, which in turn concludes the proof of \eqref{17ott1} as it contradicts the assumption $\|\boldsymbol{\phi}_n\|=1$ for every $n$.

To this aim, let us test \eqref{phi_in_perp} with $\phi_{i}^n$, so to have
\begin{equation}
\label{Step3}
\begin{split}
\|\phi_i^n\|^2=&\underbrace{\left(3\mu_i  c _i^2+\sum_{j\in I_h\atop j\neq i}\beta_{ij} c _j^2\right)\int_{\Omega}\left(PU_h^n\right)^2(\phi_i^n)^2}_{I}+2\underbrace{\sum_{j \in I_h\atop j\neq i}\beta_{ij} c _i c _j\int_{\Omega}\left(PU_h^n\right)^2\phi_j^n\phi_i^n}_{II}\\
+&\underbrace{\sum_{k\neq h}\sum_{j \in I_k}\beta_{ij} c _j^2\int_{\Omega}(PU_k^n)^2(\phi_i^n)^2}_{III}+\underbrace{\sum_{k\neq h}\sum_{j  \in I_k }2\beta_{ij} c _i c _j\int_{\Omega}PU_k^n PU_h^n\phi_j^n\phi_i^n}_{IV}\\
+&\lambda_{i,n}\|\phi_i^n\|^2+\langle h_i^n-w_i^n,\phi_i^n\rangle\,.
\end{split}
\end{equation}
Since $\boldsymbol{\lambda}_n\to0$, $\|{\bf h}_n\|\to0$, $\|{\bf w}_n\|\to0$ and $\boldsymbol{\phi}_n$ is bounded in $H_0^1(\Omega)$ uniformly on $n$,
\begin{equation}
\label{Step3_o(1)}
\lambda_{i,n}\|\phi_i^n\|^2+\langle h_i^n-w_i^n,\phi_i^n\rangle\to0\qquad\text{as }n\to+\infty\,.
\end{equation}
Moreover, recalling that $0\leq PU_h^n\leq U_h^n$ for every $h=1,\dots,q$, we have
\[
\begin{split}
\int_{\Omega}\left(PU_h^n\right)^2(\phi_i^n)^2\leq& \int_{\Omega}\left(U_h^n\right)^2(\phi_i^n)^2=\int_{\widetilde{\Omega}_{h,n}}U_{1,0}^2(\widetilde{\phi}_i^n)^2\to0\\
\int_{\Omega}\left(PU_h^n\right)^2\phi_j^n\phi_i^n=&\int_{\Omega}\left(U_h^n\right)^2\phi_j^n\phi_i^n+o(1)=\int_{\widetilde{\Omega}_{h,n}}(U_{1,0})^2\widetilde{\phi}_j^n\widetilde{\phi}_i^n+o(1)\to0
\end{split}
\]
as $n\to+\infty$ and for every $i,j\in I_h$, since $\widetilde{\phi}_i^n,\widetilde{\phi}_j^n\rightharpoonup 0$ in $\mathcal{D}^{1,2}(\mathbb{R}^n)$ and $U_{1,0}^2\in L^2(\mathbb{R}^n)$, so that
\begin{equation}
	\label{Step3_I II}
	|I|\to0\quad\text{and}\quad|II|\to0\quad\text{as }n\to+\infty\,.
\end{equation}
As for term $IV$, for every $i\in I_h$, $j\in I_k$, $h\neq k$, by H\"older and Sobolev inequalities and by \cite[Lemma A.2--A.4]{PistoiaTavares}
\[
\begin{split}
\left|\int_{\Omega}PU_k^n PU_h^n\phi_j^n\phi_i^n\right|\leq&\left(\int_{\Omega}(PU_k^n)^2(PU_h^n)^2\right)^\frac{1}{2}\left(\int_{\Omega}(\phi_j^n)^2(\phi_i^n)^2\right)^\frac{1}{2}\leq C\left(\int_{\Omega}(U_k^n)^2(U_h^n)^2\right)^\frac{1}{2}\\
\leq&C'\left(O(\delta_{h,n}^2)\int_\Omega(U_k^n)^2+O(\delta_{k,n}^2)\int_\Omega(U_k^n)^2+O(\delta_{h,n}^2\delta_{k,n}^2)\right)^\frac{1}{2}\\
\leq&C''\left(O(\delta_{h,n})\delta_{k,n}\sqrt{|\ln\delta_{k,n}|}+O(\delta_{k,n})\delta_{h,n}\sqrt{|\ln\delta_{h,n}|}+O(\delta_{h,n}\delta_{k,n})\right)\,,
\end{split}
\]
thus ensuring
\begin{equation}
	\label{Step3_IV}
	|IV|\to0\qquad\text{as }n\to+\infty\,.
\end{equation}
We are left to discuss term III. On the one hand, if for every $h=1,\dots,q$ it holds 
\[
\max\limits_{(i,j)\in I_h\times I_k\atop h\not=k}\beta_{ij}\leq0\,,
\]
then we simply have
\[ III\le0.
\]
On the other hand, if there exist $i\in I_h,j\in I_k$ with $\beta_{ij}>0$, then 
\[
\beta_{ij}\int_{\Omega}(PU_k^n)^2(\phi_{i}^n)^2\leq\beta_{ij}\left|\phi_{i}^n\right|_4^2\left|U_k^n\right|_4^2\leq C\beta_{ij}\|\phi_{i}^n\|^2\,.
\]
Let then $\overline{\beta}>0$ be a positive constant so that, whenever 
\[
\max\limits_{(i,j)\in I_h\times I_k\atop h\not=k}\beta_{ij}\leq \overline{\beta}\,,
\] 
we have
\begin{equation}
\label{Step3_III}
|III|\leq C\sum_{k\neq h}\sum_{j  \in I_k }\beta_{ij} c _j^2\|\phi_{i}^n\|^2\leq\frac{1}{2}\|\phi_i^n\|^2\,.
\end{equation}
Summing up, coupling \eqref{Step3_o(1)},\eqref{Step3_I II},\eqref{Step3_IV} and \eqref{Step3_III} with \eqref{Step3}, we conclude that $\|\phi_{i}^n\|\to0$ as $n\to+\infty$, for every $i=1,\dots,m$.

\smallskip
{\em Step 4: invertibility.} Note first that $(-\Delta)^{-1}:L^{\frac{4}{3}}(\Omega)\to H_0^1(\Omega)$ is a compact operator, so that $\bs{\mathcal L} $ restricted to $\bf K ^\perp$ is a compact perturbation of the identity. Furthermore, \eqref{17ott1} implies that $\bs{\mathcal L} $ is injective, and thus surjective by Fredholm alternative. Henceforth, it is invertible, and the continuity of the inverse operator is guaranteed by \eqref{17ott1}.

\bibliography{system}
\bibliographystyle{plain}

\end{document}